\documentclass[12pt, reqno, a4paper]{amsart}
\usepackage{amsmath,bbm}
\usepackage{amssymb}
\usepackage{lmodern}
\usepackage[hidelinks]{hyperref}

\usepackage{enumitem}
\setitemize{noitemsep,topsep=0pt,parsep=0pt,partopsep=0pt}

\newcommand{\N}{\mathbb{N}}

\newcommand{\Z}{\mathbb{Z}}

\newcommand{\diag}{\mathop{\mathrm{diag}}\nolimits}

\allowdisplaybreaks
\theoremstyle{plain}
\newtheorem{theorem}{Theorem}[section]
\newtheorem{lemma}[theorem]{Lemma}
\newtheorem{corollary}[theorem]{Corollary}

\theoremstyle{definition}

\theoremstyle{remark}
\newtheorem{remark}[theorem]{Remark}

\newcommand{\E}{{\mathbb{E}}}
\newcommand{\R}{{\mathbb{R}}}

\newcommand{\ind}{\operatorname{1}}

\renewcommand{\epsilon}{\varepsilon}
\renewcommand{\phi}{\varphi}

\newcommand{\ent}{\mathrm{ent}}

\numberwithin{equation}{section}

\begin{document}

\title[Residual dependence in Gibbs measures]{Propagation of chaos and residual dependence in Gibbs measures on finite sets}

\author[Jonas Jalowy]{Jonas Jalowy}

\author[Zakhar Kabluchko]{Zakhar Kabluchko}

\author[Matthias L\"owe]{Matthias L\"owe}
\address[Jonas Jalowy, Zakhar Kabluchko, Matthias L\"owe]{Fachbereich Mathematik und Informatik,
Universit\"at M\"unster,
Einsteinstra\ss e 62,
48149 M\"unster,
Germany}
\email[Jonas Jalowy]{jjalowy@uni-muenster.de}
\email[Zakhar Kabluchko]{zakhar.kabluchko@uni-muenster.de}
\email[Matthias L\"owe]{maloewe@math.uni-muenster.de}

\date{\today}



\newcommand{\wlim}{\mathop{\hbox{\rm w-lim}}}
\newcommand{\na}{{\mathbb N}}
\newcommand{\re}{{\mathbb R}}

\newcommand{\vep}{\varepsilon}
\newcommand{\be}{\begin{equation}}
\newcommand{\ee}{\end{equation}}

\begin{abstract}
We compare a mean-field Gibbs distribution on a finite state space on $N$ spins to that of an explicit simple mixture of product measures. This illustrates the situation beyond the so-called {\it increasing propagation of chaos}
introduced by Ben Arous and Zeitouni \cite{BAZ_chaos}, where marginal distributions of size
$k=o(N)$ are compared to product measures.
\end{abstract}

\maketitle

\section{Introduction and main results}
This note is devoted to the approximation of mean-field Gibbs measures $\mu_N$ on a finite state space,
by arguably easier measures, namely mixtures of (random) product measures.
Ben Arous and Zeituoni \cite{BAZ_chaos} show that the marginal distribution of blocks of $k=o(N)$ out of $N$ spins can be approximated well by the product measure in terms of the relative entropy. Thus, the spins become independent (or, 'chaotic'), a phenomenon that is called \emph{increasing propagation of chaos}. Here, \emph{increasing} refers to the block sizes going to infinity, as $N \to \infty$. As we will see in this note,
this cannot continue to be true, if the blocks are larger than $o(N)$.
That being the case, it is natural to ask what happens after the chaos stops to propagate and how the marginals can be approximated instead. A solution to this
problem on finite state spaces will be the main result of the present note. The resulting approximating measures $\nu_N$ of our study describe the residual dependence between the spins in the full Gibbs measure.

As the random variables we investigate are exchangeable due to their mean-field structure,
de Finetti's Theorem would imply the existence of a non-explicit infinite sequence of such random variables
as mixtures of product measures. However, our sequences are finite. For finite sequences of random variables,
a related goal as ours was pursued in \cite{BAYE23}, \cite[\S 5]{KirschSurvey} or \cite{LST07}. However, these works do not aim to connect the results to the so-called increasing propagation of chaos introduced in \cite{BAZ_chaos}, and focus on an explicit derivation of the mixing de Finetti measure of special models, in particular the (high temperature) Curie-Weiss model.

To be more precise, we will study a setup that slightly differs from the framework given in \cite{BAZ_chaos}.
Our mathematical framework is the following: We consider a general mean field model of a spin system of $N$ spins, which can
take values in a finite set $\{1,\dots,q\}$, where $q \ge 2$. 
Define the ''magnetization''
\begin{align*}
    m_N: \{1,\dots,q\}^N&\to \Delta:=\{m\in [0,1]^q:|m|=\sum_{k=1}^qm_k=1\}\\
    m_N(\sigma)&:=\Big(\frac 1 N \sum_{i=1}^N \mathbbm 1_k(\sigma_i)\Big)_{k=1,\dots,q}, \qquad \sigma=(\sigma_i)_{i=1}^N \in \{1,\dots,q\}^N
\end{align*}
which we assume to be an order parameter for the Gibbs measure.
That is, for a fixed function $F\in\mathcal C^2(\Delta)$, the Gibbs measure shall be given by
\begin{equation}\label{eq:gibbsmeasure1}
\mu_N(\sigma)=\frac 1 {Z_N} \exp\big( N F(m_N(\sigma))\big), \qquad \sigma=(\sigma_i)_{i=1}^N \in \{1,\dots,q\}^N \end{equation}
Here, the partition function $Z_N$ is given by
$$
Z_N= \sum_{\sigma \in \{1,\dots,q\}^N} \exp\big( N F(m_N(\sigma))\big).
$$
Examples of such measures are mean-field Potts models (see e.g.\ \cite{KestenSchonmann89}), $p$-spin models (see \cite{p-spin}) and especially,
the mean-field Ising model (or Curie-Weiss) model for $q=2$ and $F(x)=\beta (1-2x_1)^2/2$ for some inverse temperature $\beta>0$.
The Curie-Weiss Gibbs measure as a distribution on $\{\pm 1\}^N$ is given by $\mu_N(\sigma):= \frac 1 {Z_N}\exp\left(\frac{\beta}{2N} \sum_{i,j=1}^N \sigma_i \sigma_j\right)$
and we will study it in detail in Section \ref{sec:CW}.

Before we return to the general framework, let us forecast the most simple case of our main result:
Let $0<\beta<1$ and $\tilde\xi\sim\mathcal N (0,\frac {\beta} {1-\beta})$ be a Gaussian random variable. We will truncate $\tilde \xi$
at height $N^\delta$ for some small $\delta>0$ to obtain $\xi=\tilde\xi \mathbbm 1_{|\tilde\xi|\le N ^\delta}$. This technical truncation serves the sole purpose to
ensure that the following measures remain non-negative.  Then, Corollary \ref{cor:main1a} states that the mixed product measure
\begin{align*}
\nu_N(\sigma)=2^{-N}\E_\xi\Big[\prod_{i=1}^N \big( 1+\sigma_i\frac{\xi}{\sqrt N}\big)\Big]
\end{align*}
approximates the Curie-Weiss Gibbs measure, i.e. $H(\mu_N|\nu_N)\to 0 $ as $N\to\infty$. The guiding idea for choosing the mixing factors is to choose them in such a way that
the measure $\nu_N$ has the same asymptotic variance as the Gibbs measure $\mu_N$.

Turning back to a general Gibbs measure $\mu_N$ on the finite set $\{1,\dots,q\}$, let us now describe the setting. We postulate that $m_N$ is a sufficient statistic (or an `order parameter' as described before)
and equivalently consider the pushforward of the measure $\mu_N$ under
the mapping $\sigma \mapsto m_N(\sigma)$, i.e.\ the probability measure $\tilde \mu_N$ on the simplex $\Delta$ given by:
\[\tilde \mu_N(m)=\frac 1 {Z_N} \binom{N}{Nm_1,\dots,Nm_q} \exp\big( N F(m)\big), \qquad m \in \Delta. \]
Similarly, we may represent the partition function as
$$
Z_N=\sum_{m\in \Delta\cap \frac 1 N \Z^q}\binom{N}{Nm_1,\dots,Nm_q} \exp\big( N F(m)\big).
$$
Note that by Stirling's formula
\begin{align}\label{eq:Stirling}
\frac 1N \log \binom{N}{Nm_1,\dots,Nm_q} \sim \ent(m),
\end{align}
where for two sequences $a_N$ and $b_N$ we write $a_N \sim b_N$, if and only if $a_N/b_N \to 1$ as $N \to \infty$ and
$$\ent(m)=-\sum_{k=1}^q m_k\log(m_k)$$ is the entropy function.

Since the simplex $\Delta$ is $q-1$ dimensional and therefore $q-1$ coordinates determine the value of $m$,
in the sequel we shall restrict ourselves to the first $q-1$ coordinates of $m$ and denote
\[\hat m=(m_k)_{k=1,\dots,q-1}\in \hat\Delta:=\{\hat m\in [0,1]^{q-1}: |\hat m| \le 1\}.\]
It is well known that the asymptotics of $m_N$ under the Gibbs measure $\mu_N$ are governed
by the global maximum points $\sup_{\hat \Delta} G$ of the (up to constants) Helmholtz free energy function
\begin{align}
    G(\hat m):=F(m)+\ent(m),
\end{align}
which we shall view as a function of the first $q-1$ coordinates only.
In particular Varadhan's lemma, Theorem III.13 in \cite{FdHLarge}, shows
\be \label{eq:logZN}
\lim_{N \to \infty} \frac 1 N \log Z_N = \sup_{\hat {\mathfrak m} \in\hat \Delta} G(\hat {\mathfrak m}).
\ee
Moreover, by the tilted large deviation principle \cite[Theorem III.17]{FdHLarge}, $m_N$ satisfies a large deviation principle
under the sequence of Gibbs measures $\mu_N$ with speed $N$ and rate function
$$
I(m) = -G(\hat m)+\sup_{\hat{\mathfrak{m}} \in \hat \Delta} G(\hat{\mathfrak{m}}).
$$
We assume that the so defined $G$ has the following two properties:
\begin{enumerate}
    \item[(A)] $G$ has $p\in\N$ many isolated global maximizers $M_1,\dots,M_p$ that lie in the interior of $\hat\Delta$ and
    \item[(B)] $G$ has a non-degenerate Hessian in these maximizers, i.e. $H_j:=-\mathrm{Hess}_G(M_j)>0$ (i.e.\ is positive definite)
		for each $j=1,\dots, p$.
\end{enumerate}

Denote the coordinates of maximizer $M_j$ by $M_{j,k}$, $k=1, \ldots, q-1$ and $j=1, \ldots, p$.
With a slight abuse of notation, we write $M_{j,q}=1-|M_j|$ for $j=1,\dots, p$. Under these assumption, an application of Laplace's
method implies we may refine the asymptotics \eqref{eq:logZN} of the partition function. Even though this can, in principle, also be
derived from Theorem B in
\cite{BB90}, who treat a more general situation similar to the one in \cite{BAZ_chaos}, for completeness we give a direct proof of the following lemma
in Section 2.
\begin{lemma}\label{lem:partitionfunction}
As $N\to \infty$ it holds
\[Z_N\sim |w| \exp(N\mathrm{sup}_{\hat \Delta} G),\]
where the weight is given by $| w|=\sum_{j=1}^pw_j$ for $w_j=\big( \det(H_j)\prod_{k=1}^q M_{j,k}\big)^{-1/2}$.
\end{lemma}

In our main result we consider the $k$-marginals of the measure $\mu_N$, where $k=k(N)$ may depend on $N$ (by exchangeability
of the $(\sigma_i)_{i=1}^N$ the specific choice of the $k$ coordinates is unimportant).
This main result can be understood as a generalized form of the increasing propagation of chaos result in the
sense of \cite{BAZ_chaos,LLP10}, where $k$ is even allowed to be of size $\alpha N$, $0 \le \alpha \le 1$.
This allows us to describe the residual dependence between the spins $(\sigma_i)$ in our model.

More precisely, we wish to approximate $\mu_N$ by a mixture $\nu_N$ of product measures. The $j$'th summand ($j=1, \ldots, p$) of this mixture
is a product measure that is averaged over some additional
$q$-dimensional random vector $\xi^{(j)}=(\xi^{(j)}_k)_{k=1, \ldots q}$. To make this more precise, we define
\begin{align}\label{eq:nu}
\nu_N(\sigma):= \frac{1}{| w|}\sum_{j=1}^p w_j\E_\xi\Big[\prod_{i=1}^N\sum_{k=1}^q \Big(M_{j,k}+\frac{\xi_k^{(j)}}{\sqrt N}\Big)\mathbbm 1_{k}(\sigma_i)\Big],
\end{align}
whose pushforward under $m_N$ is given by
\begin{align}\label{eq:nutilde}
\tilde\nu_N(m):= \frac{1}{| w|}\sum_{j=1}^p w_j\binom{N}{Nm_1,\dots,Nm_q}\E_\xi\Big[\prod_{k=1}^q \Big(M_{j,k}+\frac{\xi_k^{(j)}}{\sqrt N}\Big)^{Nm_k}\Big].
\end{align}
To specify the distribution of the randomness is $\nu_N$: The mixture in $\nu_N$ is incorporated by (independent) non-degenerate $(q-1)$ Gaussian random vectors
\begin{align}\label{eq:Gaussian_distribution}
    \tilde \xi^{(j)}\sim \mathcal N_{q-1}(0,\Sigma),\quad \Sigma=\Sigma^{(j)}=H_j^{-1}-\diag(M_j)+M_jM_j^T.
\end{align}
These Gaussian random vectors are truncated at a Euclidean norm of $N^{\delta}$ to obtain
$$\xi^{(j)}=\tilde\xi^{(j)}\mathbbm 1 _{\lVert\tilde\xi^{(j)}\rVert\le {N^{\delta}}}$$
for some $\delta>0$ to be chosen later. This truncation ensures that $\nu_N$ is positive (and smaller than $1)$.
We complete $\xi^{(j)}_k, k=1, \ldots q-1$ by a $q$-th coordinate  $\xi_q^{(j)}=-\sum_{k=1}^{q-1}\xi_k^{(j)}$ such that $\sum_{k=1}^q\xi_k^{(j)}=0$.
Note that each summand in \eqref{eq:nu} only differs from the product measure $M_j^{\otimes N}$ on $\{1,\dots,q\}^N$ by artificially mixing over $\xi^{(j)}_k, k=1, \ldots q$
and we shall see that
it is mixed precisely in such a way that the covariance matrix of $\tilde\nu_N$ matches asymptotically the covariance of $\tilde \mu_N$. E.g.~the covariance matrix in the
multivariate CLT for the multinomial distribution (that is the push-forward of $M_j^{\otimes N}$ under $m_N$) is given by $\diag(M_j)-M_jM_j^T$. Thus, the propagation of chaos
stops for marginals of size $k\sim \alpha n$ as we shall prove in Corollary \ref{cor:chaos_stops}, generalizing the ideas of \cite{JKLM}.

We will quantify the distance between probability measures $\varrho$ and $\pi$ on the same probability space in terms of relative entropy or Kullback-Leibler information defined by
\be \label{eq:KL}
H(\varrho|\pi):= \begin{cases}
\int \log\left(\frac{d\varrho}{d\pi}\right)d\varrho & \mbox{if } \varrho \ll \pi \\
\infty & \mbox{otherwise.}
\end{cases}
\ee
With these definitions, our central result reads as follows:
\begin{theorem}\label{theorem:main_gen}
As $N \to \infty$ the mixture $\nu_N$ approximates the Gibbs measure $\mu_N$ in relative entropy, i.e.\
\[
H(\mu_N|\nu_N)\to 0 \qquad \text{as well as}\qquad H(\nu_N|\mu_N)\to 0.\]
\end{theorem}

\begin{remark}
This answers the motivating question of approximating the $k$-marginals $\mu_N^{(k)}$ by $\nu_N^{(k)}=\nu_k$ even for $k/N\to\alpha$. Indeed, by Jensen's inequality it holds $H(\mu_N^{(k)}|\nu_k)\to 0$.
\end{remark}

\begin{remark}\label{rem:quantification}
It can be deduced from the proof that the rate of convergence in the above approximation is of order $\mathcal O(N^{-1/2+\epsilon})$ for any small $\epsilon>0$. We did not aim to optimize this rate quantitatively.
\end{remark}

\begin{remark}\label{rem:generalization}
Assumption (B) can be weakened to the following. For each $j=1,\dots,p$ let $G\in\mathcal C^\ell$ for some $\ell\ge 2$ such that its first $\ell-1$ derivatives vanish at a maximizer $M_j$ and
$$G(\hat m)=G(M_j)+G^{\{\ell\}}(\hat m -M_j)+o(\lVert\hat m-M_j\rVert^\ell)$$
for some function $G^{\{\ell\}}$ that is positive away from zero. Then, Theorem \ref{theorem:main_gen} continues to hold for a mixed measure $\nu_N$ similar to \eqref{eq:nu}, where the truncated Gaussian $\xi^{(j)}_k/\sqrt N$ is replaced by $\xi^{(j)}_k/N^{1/\ell}$ and the random variables $\xi^{(j)}_k$ have a distribution
with tails of order $\exp(-c\lVert x\rVert^\ell)$. Our assumption (B) corresponds to $\ell=2$ and the well-known example for $\ell=4$ is given by the Curie Weiss model at critical temperature, which we shall explicitly work out in Theorem \ref{theo:mainCW1b} below. In the general degenerate situation, there is no Central Limit Theorem for $m_N$ and therefore also another approximation of the Gibbs measure by mixtures of product measures must be found. However, these depend very much on the specific form of the Hessian in the maximizers.
Nevertheless, the proof of the general claim is an adjustment of the proof we shall present, however for reasons of simplicity and accessibility, we leave the details to the (ambitious) reader.
\end{remark}

We organize the rest of this paper in the following way: In Section 2 we will prove Theorem \ref{theorem:main_gen}. In Section 3 we will comment on the relation of our result to the (by now classical) increasing propagation of chaos in the sense of \cite{BAZ_chaos}.
In Section 4, we will treat a specific case, the
Curie-Weiss model, where an approximation of the Gibbs measure by mixtures of product measures can explicitly be given even in the case of critical inverse temperature $\beta$.

\section{Proof of Theorem \ref{theorem:main_gen}}\label{sec:General}

Let us begin with a straightforward proof of Lemma \ref{lem:partitionfunction}.

\begin{proof}
Consider the open $\vep$-neighborhoods $\mathcal{N}_j$ of the maximizers $M_j$
$$
\mathcal{N}_j:=\left\{m\in \Delta: \sqrt{\sum_{i=1}^q \left(m_i-M_{i,j}\right)^2}< \vep\right\}, \quad j=1,\ldots,p
$$
where $\vep>0$ is chosen so small that $\mathcal{N}_j\subset \Delta$ for all $j=1,\ldots,p$. The above mentioned large deviations
principle for $m_N$ under the Gibbs measure shows that the $m \notin \bigcup_{j=1}^p \mathcal{N}_j$ have asymptotically
negligible contribution to $Z_N$.

Now, for $m \in \bigcup_{j=1}^p \mathcal{N}_j$ by Stirling's formula
$$
\log (n!)= n \log n - n +\frac 12 \log(2 \pi) + \frac 12 \log ( n  )+\mathcal{O}(1/n),
$$
we obtain a refinement of
\eqref{eq:Stirling}, i.e.
\begin{align}\label{eq:Stirling_exakt}
\frac 1N \log \binom{N}{Nm_1,\dots,Nm_q} \sim \ent(m) +\mathcal{O}\left(\frac 1N \right)
\end{align}
with an explicit error term that is uniform on $\bigcup_{j=1}^p \mathcal{N}_j$.
Thus with the definition $\mathcal{N}:=\bigcup_{j=1}^p \mathcal{N}_j$ and again with
a $\mathcal{O}\left(\frac 1N \right)$-term that is uniform on $\mathcal{N}$
\begin{align*}
Z_N&=\sum_{m\in \Delta\cap \frac 1 N \Z^q}\binom{N}{Nm_1,\dots,Nm_q} \exp\big( N F(m))\big)\\
&=\sum_{m\in \mathcal{N}\cap \frac 1 N \Z^q}\binom{N}{Nm_1,\dots,Nm_q} \exp\big( N F(m))\big)(1+o(1))\\
&=\sum_{m\in \mathcal{N} \cap \frac 1 N \Z^q}\exp\big(N\big(F(m)+\ent(m)\big)+\mathcal O\left(\frac 1N\right)\big)\frac{1}{\sqrt{(2\pi N)^{q-1} \prod_{k=1}^q m_k }}(1+o(1))\\
&\sim\int_{\mathcal{N}}\exp\big(NG(x)\big)\frac{N^{(q-1)/2}}{\sqrt{(2\pi)^{q-1} \prod_{k=1}^q x_k }} dx\\
&\sim\int_{\hat\Delta}\exp\big(NG(x)\big)\frac{N^{(q-1)/2}}{\sqrt{(2\pi)^{q-1} \prod_{k=1}^q x_k }} dx,
\end{align*}
where we used Riemann approximation of the integral and we wrote $x_q=1-\sum_{k=1}^{q-1}x_k$. Now, a standard application of Laplace's method shows that
\begin{align*}
Z_N\sim \sum_{j=1}^p  \exp\big(NG(M_j)\big) \frac 1 {\sqrt{H_j \prod_{k=1}^q M_{j,k}}}= |w| \exp(N\mathrm{sup}_{\hat \Delta} G)
\end{align*}
as claimed.
\end{proof}

The crucial part of the proof of Theorem \ref{theorem:main_gen} is to show that the relative density
$\frac{d\tilde\mu_N}{d\tilde{\nu_N}}(m)$ converges to $1$ for all important arguments $\hat m$, i.e.\ for those $\hat m$ that are
close to a maximizer $M_j$.
This will be done in the following lemma.
\begin{lemma}\label{lem:dichtequotient}
Let $0<\delta<1/6$ and fix $J=1,\dots,p$. Uniformly in $m\in\Delta$ with $\lVert M_J-\hat m\rVert=o (N^{-1/2+\delta})$ it holds \[\frac{d\tilde\mu_N}{d\tilde \nu_N}(m)=1+o(1).\]
\end{lemma}

\begin{proof}
For notational simplicity, we will skip the superscript of $\xi=\xi^{(j)}$ from now on. Let us fix $J=1,\dots,p$ and take $m\in\Delta$ such that
$\lVert M_J-\hat m\rVert = \mathcal{O} (N^{-1/2+\delta})$.
For convenience we will consider the inverse quotient $\frac{d\tilde\nu_N}{d\tilde \mu_N}(m)$.
It follows from an application of Lemma \ref{lem:partitionfunction}, the fact that $ M_J$ maximizes the function $G$ and by
the definition of $\ent(m)$ that
\begin{align}
    \frac {\tilde\nu_N}{\tilde \mu_N}(m)&=\frac{Z_N \sum_{j=1}^p w_j\E_\xi\big[\prod_{k=1}^q(M_{j,k}+\xi_k/\sqrt N)^{Nm_k}\big]}{|w|\exp(NF(m))}\nonumber\\
    &\sim \sum_{j=1}^p w_j\exp\big(N(\mathrm{sup}_{\hat \Delta} G-F(m))\big)\E_\xi\big[\prod_{k=1}^q(M_{j,k}+\xi_k/\sqrt N)^{Nm_k}\big]\nonumber\\
		&=\sum_{j=1}^p w_j\exp\big(N(G(M_J)-F(m)-\ent(m))\big)\E_\xi\Big[\exp\Big( N\sum_{k=1}^qm_k\log\big(\frac{M_{j,k}+\xi_k/\sqrt N}{m_k}\big)\Big)\Big]\nonumber\\
    &=\sum_{j=1}^p w_j\exp\big(N(G(M_J)-G(\hat m))\big)\E_\xi\Big[\exp\Big( N\sum_{k=1}^qm_k\log\big(\frac{M_{j,k}+\xi_k/\sqrt N}{m_k}\big)\Big)\Big]\label{eq:dichtequotientexpexp}
    \end{align}
We first argue that in the (outer) sum in the last line only the summand $j=J$ matters.
Indeed, the argument in the exponential under the expectation $\E_\xi$
can be rewritten (up to the factor $-N$) as a relative entropy between finite distributions on $\{1,\dots, q\}$:
\[
\sum_{k=1}^qm_k\log\big(\frac{M_{j,k}+\xi_k/\sqrt N}{m_k}\big)=-H(m|M_j+\xi/\sqrt N)\le 0
\]
since relative entropy is non-negative by Jensen inequality. In particular, as is well known, we have
$\sum_{k=1}^qm_k\log\big(\frac{M_{j,k}+\xi_k/\sqrt N}{m_k}\big)=0$
if and only if $\hat m=M_j+\xi/\sqrt N$. Hence for $j\neq J$ this term is bounded
from above by some strictly negative constant $-c<0$, if $N$ is so large that the $N^{-1/2+\delta}$-neighborhoods of the $M_j$
do not overlap.

    Regarding the first exponential in \eqref{eq:dichtequotientexpexp}, our assumptions on $G$ allow for a Taylor approximation of $G(\hat m)$ around $M_J$ with vanishing first order term
    \begin{align}
        N(G(M_J)-G(\hat m))&=-\frac N 2(M_J-\hat m)^T\mathrm{Hess_G}(M_J)(M_J-\hat m)+\mathcal O( N\lVert\hat m -M_J\rVert ^3)\nonumber\\
        &=\frac N 2(M_J-\hat m)^TH_J(M_J-\hat m)+o(N^{-1/2+3\delta}).\label{eq:HessG}
    \end{align}
If $\hat m$ is not close to $M_J$, which is in particular the case, if $j\neq J$,
then the first summand on the right hand side is $\mathcal{O}\left(N (N^{-1/2+\delta})^2\right)=\mathcal{O}\left(N^{2\delta}\right)$ while
the second summand on the right hand side is negligible by our choice of $\delta$.
Hence the whole summand of \eqref{eq:dichtequotientexpexp} is of negligible order $\mathcal O( \exp(N^{2\delta}-cN))$.

    Now for $j=J$, another Taylor approximation of the logarithm in the vicinity of $1$ yields
    \begin{align}
        &N\sum_{k=1}^qm_k\log\big(\frac{M_{J,k}+\xi_k/\sqrt N}{m_k}\big)\nonumber\\
        =&N\sum_{k=1}^qm_k\log\big(1+ \frac{M_{J,k}-m_k+\xi_k/\sqrt N}{m_k}\big)\nonumber\\
        =&N\sum_{k=1}^q ( M_{J,k}-m_k +\xi_k/\sqrt N)-\frac N 2 \sum_{k=1}^q \frac{( M_{J,k}-m_k +\xi_k/\sqrt N)^2}{m_k}+\mathcal O(N^{-1/2+3\delta}),\label{eq:logtaylor}
    \end{align}
    where we used the uniform cutoff $\xi_k=\mathcal O (N^\delta)$ for all $k$ by definition and $m_k^{-1}=\mathcal O (1)$ by our Assumption (A).
		
In order to simplify notation in the following, let $\mathbf 1$ be the $(q-1)$-vector consisting of all ones, $\diag(v)$ be the
$(q-1)\times (q-1)$-diagonal matrix of a $(q-1)$-vector $v$ and $I=\diag \mathbf 1$ be the
$(q-1)\times (q-1)$-identity matrix. E.g. by $\mathbf 1^T\hat m=1-m_q$,
$\mathbf 1^T\hat M_J=1-M_{J,q}$  and
$\mathbf 1^T \xi=-\xi_q$ and for the first sum in \eqref{eq:logtaylor} above we obtain
$$\sum_{k=1}^q ( M_{J,k}-m_k +\xi_k/\sqrt N)=0.$$
We split the remaining sum into the summands $k<q$ and $k=q$ to obtain
 \begin{align*}
        &N\sum_{k=1}^qm_k\log\big(\frac{M_{J,k}+\xi_k/\sqrt N}{m_k}\big)\\
        &=-\frac 1 2 \big( \sqrt N (M_J-\hat m)+\xi\big)^T \diag({\hat m})^{-1}\big( \sqrt N (M_J-\hat m)+\xi\big)\\
				&\hskip2.5cm-\frac 1 {2m_q}\big(\mathbf 1 ^T \big( \sqrt N (M_J-\hat m)+\xi\big)\big)^2+\mathcal O(N^{-1/2+3\delta})\\
       &=-\frac 1 2 (a+\xi)^TA (a+\xi)+\mathcal O(N^{-1/2+3\delta}),
        \end{align*}
where we define
$$
A=\diag(\hat m) ^{-1}+\frac1{m_q}\mathbf 1\mathbf 1^T
$$
and $a=\sqrt N (M_J-\hat m)$. Note that $A$ is positive definite for $N$ sufficiently large, although $a$ can be negative of order $N^\delta$.
Together with \eqref{eq:HessG}, we have obtained that \eqref{eq:dichtequotientexpexp} is given by
\begin{align}
    \frac {\tilde\nu_N}{\tilde \mu_N}(m)&=w_J\exp(\tfrac 1 2 a^TH_Ja+o(1)) \E_\xi\big[\exp\big(  -\tfrac 1 2 (a+\xi)^TA (a+\xi) \big)\big].\label{eq:dichtequotient_step1}
    \end{align}
Now, recall that we assume that $\xi=\tilde \xi \mathbbm 1_{B_{N^\delta}(0)}$ for a truncated Gaussian vector $\tilde\xi$ whose density
converges to the density of a non-degenerate Gaussian distribution $\mathcal N_{q-1}(0,\Sigma)$, uniformly on $B_{N^\delta}(0)$. Hence
we may now split the integration into $B_{N^\delta}(0)^c$ and $B_{N^\delta}(0)$, and perform the Gaussian integral by completing the
square
    \begin{align}\label{eq:quadr_erg}
    &\E_\xi\big[\exp\big(  -\xi^TAa-\frac 1 2 \xi^TA\xi  \big)\big]\nonumber\\
    =&\frac 1{\sqrt{\det(2\pi\Sigma)}}\int_{B_{N^\delta}(0)}\exp\big(  - x^TAa-\frac 1 2 x^T(A+\Sigma^{-1})x  \big)dx+\mathcal N_{q-1}(0,\Sigma)\big( B_{N^\delta}(0)^c\big)\\
    =&\exp\big(\tfrac 1 2 a^TA(A+\Sigma^{-1})^{-1} Aa \big)\nonumber\\
    &\frac 1{\sqrt{\det(2\pi\Sigma)}}\int_{B_{N^\delta}(0)}\exp\big(-\frac 1 2 (x+(A+\Sigma^{-1})^{-1}Aa)^T(A+\Sigma^{-1})(x+(A+\Sigma^{-1})^{-1}Aa)  \big)dx    +o(1)\nonumber\\
    =&\exp\big(\tfrac 1 2 a^TA(A+\Sigma^{-1})^{-1} Aa \big) \sqrt{\det(\Sigma A+I)^{-1}}+o(1)\nonumber,
    \end{align}
    where we used the well known Gaussian tail estimates and the full Gaussian measure of
		$\mathcal N (-(A+\Sigma^{-1})^{-1}Aa,(A+\Sigma^{-1})^{-1})$.
		Indeed, this approximation works if the integration area covers the mean by at least the standard deviation,
		that is if $\lVert (A+\Sigma^{-1})^{-1/2}Aa\rVert=o(N^\delta)$ or equivalently our assumption
		$\lVert M_J-\hat m\rVert=o (N^{-1/2+\delta})$.

It follows
\begin{align*}
   \frac {\tilde\nu_N}{\tilde \mu_N}(m)=&w_J\sqrt{\det(I+\Sigma A)^{-1}}
    \exp\Big(\frac 1 2 a^T\big(H_J-A+A(A+\Sigma^{-1})^{-1} A\big)a +o(1)\Big)\\
    &+o\big(\exp(\tfrac 1 2 a^T(H_J -A)a)\big).
    \end{align*}
    It remains to evaluate the matrices, in particular we will check that $$H_J-A+A(A+\Sigma^{-1})^{-1} A$$ has small norm from we
		also obtain that the remainder is negligible. It is not difficult to see that
		\be \label{eq:matrix1}
		-A+A(A+\Sigma^{-1})^{-1}A=-A+((A+\Sigma ^{-1})-\Sigma^{-1})(A+\Sigma^{-1})^{-1}A =-(A^{-1}+\Sigma )^{-1}
		\end{equation}
		and another short calculation (using $\mathbf 1 m_q=\mathbf 1-\mathbf 1 \mathbf 1^T \hat m$) yields
		$$A^{-1}= \diag \hat m -\hat m \hat m^T.$$
        We would like to replace $A^{-1}$ with
		$$\tilde A^{-1}=\diag M_J-M_JM_J^T \quad \text{for  }\tilde A=\diag M_J ^{-1}+\frac1{M_{J,q}}\mathbf 1\mathbf 1^T.$$
		Then the deviation in Hilbert Schmidt norm is at most
    \begin{align*}
        \rVert A^{-1}-\tilde A^{-1}\lVert_{HS}= \lVert(\diag (\hat m- M_J)  +\hat m (\hat m -M_J)^T+(\hat m-M_J)M_J^T \rVert_{HS}=\mathcal O(N^{-1/2+\delta}).
    \end{align*}
    Since again $\tilde A^{-1}=\diag M_J-M_JM_J^T$, we can write $\Sigma =H_J^{-1}-\tilde A^{-1}$ by definition \eqref{eq:Gaussian_distribution}, therefore, by using \eqref{eq:matrix1} for the first equality,
    \begin{align*}
    H_J-A+A(A+\Sigma^{-1})^{-1} A&=H_J-(H_J^{-1}-\tilde A^{-1}+A^{-1})^{-1}\\
    &=H_J-H_J\big(I+(A ^{-1}-\tilde A^{-1})H_J\big)^{-1}\\
        &=H_J(\tilde A^{-1}-A^{-1})H_J\big(I+(A ^{-1}-\tilde A^{-1})H_J\big)^{-1}=\mathcal O (N^{-1/2+\delta}).
    \end{align*}
    Similarly, observe that
    \[\sqrt{\det( I+\Sigma A)}\sim \sqrt{\det(I+\Sigma  \tilde A)}=\sqrt{\det(H_J^{-1}  \tilde A)}=\sqrt{\det (H_J^{-1})\prod_{k=1}^q M_{J,k}^{-1}}=w_J>0.\]
    Combining all our estimates, we conclude with our claim.
        \end{proof}

A closer inspection of the proof also verifies that the error of approximation is at most $\mathcal O (N^{-1/2+3\delta})$, as mentioned in Remark \ref{rem:quantification}.

\begin{proof}[Proof of Theorem \ref{theorem:main_gen}]
Note that due to the fact, that $m_N$ is a sufficient statistic for both, the Gibbs measure $\mu_N$ as well as the mixed measure $\nu_N$, it suffices to show that
\[
H(\tilde \mu_N|\tilde \nu_N)\to 0 \qquad \text{as well as}\qquad H(\tilde \nu_N|\tilde\mu_N)\to 0.\]
For the first of these assertions choose a constant $K>1$ and write
\begin{multline*}
H(\tilde \mu_N|\tilde \nu_N)\\= \sum_{j=1}^p\sum_{m: \lVert M_j-\hat m\rVert\le K N^{-1/2+\delta}}
\log\left(\frac{\tilde \mu_N(m)}{\tilde \nu_N(m)}\right)\tilde \mu_N(m)+
\sum_{ \stackrel{m:\lVert M_j-\hat m\rVert > K N^{-1/2+\delta}}{\forall j=1, \ldots p}}
\log\left(\frac{\tilde \mu_N(m)}{\tilde \nu_N(m)}\right)\tilde \mu_N(m)
\end{multline*}
where we choose $\delta$ slightly smaller than in the proof of Lemma \ref{lem:dichtequotient}.
For the first summand above we know from Lemma \ref{lem:dichtequotient} that
$$
\log\left(\frac{\tilde \mu_N(m)}{\tilde \nu_N(m)}\right) \to 0
$$
uniformly in $j$ and $m$ such hat $\lVert M_j-\hat m\rVert\le K N^{-1/2+\delta}$. Hence this summand converges to $0$.
On the other hand, for all $m$ in the second sum we employ a similar consideration as in the proof
Lemma \ref{lem:dichtequotient}. Again using Stirling's formula, we see that there is a universal constant $C_0>0$ depending only on $q$
such that
$$
 \binom{N}{Nm_1,\dots,Nm_q} \le C_0 \exp\left(N \ent(m)\right)
$$
for all $m \in \Delta$.
Therefore
\begin{align}\label{eq:mu_MDP}
\tilde \mu_N(m) &= \frac 1 {Z_N} \binom{N}{Nm_1,\dots,Nm_q} \exp\big( N F(m))\big) \nonumber \\&\le
\frac{C_0\exp\left(N(F(m)+\ent(m))\right)}{|w| \exp (N\mathrm{sup}_{\hat \Delta} G)}\nonumber\\
&\le C_1\exp\left(N(G(\hat m)-\sup_{\hat \Delta} G)\right)
\end{align}
for some other constant $C_1>0$ 
and all large enough $N$.
Now we can find $\vep >0$ and a sufficiently small neighborhood $V_j$
of each of the $M_j$ such that $G$ is concave on each of the $V_j$
$$
\sup_{\hat m \notin \bigcup_{j=1}^p V_j} G(\hat m)-\sup_{\hat \Delta} G \le -\vep.
$$
Then for $\hat m \in V_j$ for one $j=1, \ldots, p$ but $\lVert M_j-\hat m\rVert > K N^{-1/2+\delta}$ we obtain
$$
C_1\exp\left(N(G(m)-\sup_{\hat \Delta} G)\right)\le C_1 \exp\left(-N^{2\delta} + o(1)\right)\le
C_2 \exp\left(-N^{2\delta}\right)
$$
for some $C_2>0$ by concavity. On the other hand, for $m \notin \bigcup_{j=1}^p V_j$ we simply get
$$
C_1\exp\left(N(G(m)-\sup_{\hat \Delta} G)\right)\le C_1 \exp\left(-N\vep\right).
$$
Therefore, it follows $\tilde \mu_N(m)\le C_3 \exp(-N^{2\delta})$ for $N$ sufficiently large and it remains to bound the logarithmic factor $\log(\tilde\mu_N(m)/\tilde\nu_N(m))$.
To this end, note that due to the truncation of $\xi$, for $\epsilon>0$, it holds $\xi_k/\sqrt N>-\epsilon$ for $N$ sufficiently large. By assumption (A), we obtain
$$
\binom{N}{Nm_1,\dots,Nm_q}\prod_{k=1}^q \Big(M_{j,k}+\frac{\xi_k}{\sqrt N}\Big)^{Nm_k} \ge\prod_{k=1}^q\Big(M_{j,k}-\epsilon\Big)^{N}\ge e^{-C_4 N}
$$
for some $C_4>0$, uniformly in $\xi$ and $m$.
This shows that
$$
\sum_{\stackrel{m: \lVert M_j-\hat m\rVert > K N^{-1/2+\delta}}{\forall j=1, \ldots p}}
\log\left(\frac{\tilde \mu_N(m)}{\tilde \nu_N(m)}\right)\tilde \mu_N(m)
\le C_5 N^{q+1} \exp(-N^{2\delta'}) \to 0
$$
as $N \to \infty$ and therefore that $H(\tilde \mu_N|\tilde \nu_N)\to 0$.

For the second statement we write
\begin{multline*}
H(\tilde \nu_N|\tilde \mu_N)\\= \sum_{j=1}^p\sum_{m: \lVert M_J-\hat m\rVert\le K N^{-1/2+\delta}}
\log\left(\frac{\tilde \nu_N(m)}{\tilde \mu_N(m)}\right)\tilde \nu_N(m)+
\sum_{ \stackrel{m:\lVert M_j-\hat m\rVert > K N^{-1/2+\delta}}{\forall j=1, \ldots p}}
\log\left(\frac{\tilde \nu_N(m)}{\tilde \mu_N(m)}\right)\tilde \nu_N(m).
\end{multline*}
Again the first of the two sums on the right hand side converge to $0$ by Lemma \ref{lem:dichtequotient}.
For the other sum this time observe that
$$
\tilde \mu_N(m) \ge C \exp\left(N\left(\inf_{m \in \Delta} G(\hat m)-\sup_{m \in \Delta} G(\hat m)\right)\right).
$$
As $\Delta$ is compact and $F$ and $\ent$ are continuous
$$-\infty< \inf_{m \in \Delta} G(\hat m)-\sup_{m \in \Delta} G(\hat m))<0.$$
Thus
$\log\left(\frac{\tilde \nu_N(m)}{\tilde \mu_N(m)}\right) \le CN$ for some constant $C$.
On the other hand, by Stirling's formula and following
the proof of Lemma \ref{lem:dichtequotient} for all $m$ with $\lVert M_J-\hat m\rVert > K N^{-1/2+\delta}\forall j=1, \ldots p$
we have
\begin{align}
\tilde \nu_N(m)&= \frac{1}{| w|}\sum_{j=1}^p w_j\binom{N}{Nm_1,\dots,Nm_q}\E_\xi\Big[\prod_{k=1}^q
\Big(M_{j,k}+\frac{\xi_k}{\sqrt N}\Big)^{Nm_k}\Big]\nonumber\\
&=\frac{Q(N)}{| w|}\sum_{j=1}^p w_j \E_\xi\Big[\exp\Big( N\sum_{k=1}^qm_k\log\big(\frac{M_{j,k}+\xi_k/\sqrt N}{m_k}\big)\Big)\Big]\nonumber\\
&= \frac{Q(N)}{| w|}\sum_{j=1}^p w_j \E_\xi\Big[\exp\Big(-N H(m|M_j+\xi/\sqrt N)\Big)\Big] \label{eq:bound_nu_tilde}
\end{align}
for some function $Q$ bounded by a polynomial. Now, for each fixed realization of $\xi$ the function
$m \mapsto H(m|M_j+\xi/\sqrt N)$ attains it minimum value $0$ for
$m= M_j+\xi/\sqrt N$ and it grows quadratically in a neighborhood of this minimum.
By construction for realizations of $\xi$ and all $\hat m$ such that $\lVert M_j-\hat m\rVert > K N^{-1/2+\delta}$ we have that
$$
\lVert M_j+\xi/\sqrt N-\hat m\rVert > (K-1) N^{-1/2+\delta}
$$
and $K-1>0$.

Therefore we obtain for $\hat m$ such that $\lVert M_j-\hat m\rVert > K N^{-1/2+\delta}$
$$
\tilde \nu_N(m) \le \frac{Q(N)}{| w|}\sum_{j=1}^p w_j \exp\left(-N K' N^{-1+2\delta}\right)
\le \exp\left(-N^{2 \delta'}\right)
$$
for some $0<\delta'<\delta$.
Thus
$$
\sum_{m: \stackrel{\lVert M_j-\hat m\rVert > K N^{-1/2+\delta}}{\forall j=1, \ldots p}}
\log\left(\frac{\tilde \nu_N(m)}{\tilde \mu_N(m)}\right)\tilde \nu_N(m) \le C N^2 \exp\left(-N^{2 \delta'}\right)
$$
showing that also $H(\tilde \nu_N|\tilde \mu_N)\to 0$.
\end{proof}

\begin{remark}
An inspection of the above proof shows that we used the assumption that $\xi$ is a vector of truncated {\it{Gaussian}} random variables
only in the proof of Lemma \ref{lem:dichtequotient} and specifically in \eqref{eq:quadr_erg} in order to obtain explicit bounds
on the integrals outside of $B_{N^\delta}(0)$ and to be able to complete the square in the exponent. Actually, this would also be
possible, if the distribution of $\xi$ converges ''sufficiently well'' to the Gaussian distribution assumed in Theorem
\ref{theorem:main_gen} (and the $\xi$ are truncated on a level $N^\delta$ again). One class of such distributions would be
distributions with a Lebesgue-density that converges uniformly to the Gaussian density given in \eqref{eq:Gaussian_distribution}. For instance, this property is satisfied by $\xi$, if we choose a $\beta$-distribution $B(N/2,N/2)$ for the coordinates $\xi_k$.
This latter example was the one that motivated Theorem \ref{theorem:main_gen}: If we specialize to the Curie-Weiss model as in
Section \ref{sec:CW}, then, due to the ferromagnetic interaction, one could think of the number of $+$-spins there as a behaving like a
Polya urn when $N$ gets large: The larger the proportion of positive spins under the first $N$ spins is, the more likely will it be
that also spin number $N+1$ is positive. This special case has been addressed in \cite{JKLM} in terms of the total variation distance instead of relative entropy and via a local limit theorem. Let us remark, that the scaling of $\tilde\mu_N$ and of the analogous push-forward measure in \cite{JKLM} differ, however this has no effect on the total variation distance (which is scale invariant).
\end{remark}

\section{Classical propagation of chaos}
In this section we want to study the relation of Theorem
\ref{theorem:main_gen} to the usual propagation of chaos, more precisely
we want to see that in the situation where $G$ has a unique non-degenerate maximizer our central result Theorem
\ref{theorem:main_gen} implies the increasing propagation of chaos result, Theorem 1 in \cite{BAZ_chaos}.
\begin{corollary}\label{cor:usual_chaos}
Assume that $G$ has a unique non-degenerate maximizer $M=M_1$ that lies in the interior of $\hat \Delta$.
Then Theorem \ref{theorem:main_gen}
implies increasing propagation of chaos with respect to product measure $\rho^N$ on
$\{1, \ldots, q\}^N$ where
$$
\rho(j)= M_{1,l} \quad \mbox{for all } l=1, \ldots q.
$$
More precisely, there exists a constant $c>0$ such that for any $k\le N$ the $k$'th marginal $\mu_N^{(k)}$ satisfies
\begin{align*}
H(\mu_N^{(k)}\mid\rho^k)\le c \frac{k}{N}.
\end{align*}
\end{corollary}

\begin{proof}
Using Csiszar's inequality \cite[Equation (2.10)]{csiszar} we obtain
\begin{equation}\label{eq:entropie_ungl}
H\left(\mu_{N}^{(k)}\big| \rho^{k}\right) \le \frac{k}{N} H\left(\mu_{N}\big| \rho^N \right).
    \end{equation}
Hence it remains to bound
\begin{align}\label{eq:split}
H(\mu_N\mid \rho^N)=H(\mu_N |\nu_N)+\int \log\frac {d\nu_N}{d\rho^N}d\mu_N,
\end{align}
where the first term on the right hand side vanishes according to Theorem \ref{theorem:main_gen}.
Similar to Lemma \ref{lem:dichtequotient}, $m_N$ is an order parameter for $\rho^N$
with
$$
\tilde \rho^N(m):=\rho^N(m_N=m)=\binom{N}{Nm_1,\dots,Nm_q} \prod_{l=1}^q M_{1,l}^{N m_l}.
$$
Using that in our situation $p=1$ and therefore $w_1=|w|=1$ this immediately brings us to
\begin{align*}
\frac {d\tilde\nu_N}{d\tilde\rho^N}(m)&=\frac{\frac{1}{| w|}\sum_{j=1}^p w_j\binom{N}{Nm_1,\dots,Nm_q}\E_\xi\Big[\prod_{l=1}^q
\Big(M_{j,l}+\frac{\xi_l}{\sqrt N}\Big)^{Nm_l}\Big]}{\binom{N}{Nm_1,\dots,Nm_q} \prod_{l=1}^q M_{1,l}^{N m_l}}
\\
&= \E_\xi\Big[\exp\Big( N\sum_{l=1}^q m_l\log\Big(\frac{M_{1,l}+\frac{\xi_l}{\sqrt N}}{m_l}\Big)
-N\sum_{l=1}^q m_l \log  \frac{M_{1,l}}{m_l}\Big)\Big]\\
&= \E_\xi\Big[\exp\Big( N\sum_{l=1}^q m_l\log\Big(1+\frac{M_{1,l}-m_l+\frac{\xi_l}{\sqrt N}}{m_l}\Big)
-N\sum_{l=1}^q m_l \log \Big(1+\frac{M_{1,l}-m_l}{m_l}\Big)\Big)\Big]\\
&= \E_\xi\Big[\exp\Big( N\big\{\sum_{l=1}^q \big( M_{1,l}-m_l+\frac{\xi_l}{\sqrt N}\big)-\big(M_{1lk}-m_l\big) \\
&\qquad -\frac 12 \frac{\big(M_{1,l}-m_l+\frac{\xi_l}{\sqrt N}\big)^2}{m_l}+\frac 12 \frac{\big(M_{1,l}-m_l\big)^2}{m_l}+
\mathcal{O}(N^{3 \delta-1/2})\Big\}\Big)\Big]
\end{align*}
where we again used Stirling's formula for the multinomial coefficient, a Taylor expansion of
$\log$ together with the truncation of the $\xi$-variables at a
height of $N^{\delta}$ with $0<\delta<1/6$.
Now the linear term $N\sum_{l=1}^q \big( M_{1,l}-m_l+\frac{\xi_l}{\sqrt N}\big)-\big(M_{1,l}-m_l\big) $ vanishes,
since $\sum_{l=1}^q \xi_l=0$. As for the remaining term
$$
\E_\xi\Big[\exp\Big(-\frac N2 \frac{\big(M_{1,l}-m_l+\frac{\xi_l}{\sqrt N}\big)^2}{m_l}+\frac N2 \frac{\big(M_{1,l}-m_l\big)^2}{m_l}
\Big)\Big]
$$
one shows that this quantity is bounded uniformly in $N$ and for all $m$ such that $\lVert M_J-\hat m\rVert=\mathcal O (N^{-1/2+\delta})$
as in the proof of Lemma \ref{lem:dichtequotient}. More precisely, with notation introduced there
\begin{align*}
&\E_\xi\Big[\exp\Big(-\frac N2 \frac{\big(M_{1,l}-m_l+\frac{\xi_l}{\sqrt N}\big)^2}{m_l}+\frac N2 \frac{\big(M_{1,l}-m_l\big)^2}{m_l}
\Big)\Big]\\
&= \E_\xi\big[\exp\big(  -\xi^TAa-\frac 1 2 \xi^TA\xi  \big)\big]\\
&= \exp\big(\tfrac 1 2 a^TA(A+\Sigma^{-1})^{-1} Aa \big) \sqrt{\det(\Sigma A+I)^{-1}}+o(1) \le \exp(C)
\end{align*}
for some appropriate $C$.

Therefore, taking some arbitrary constant $c>0$,
splitting the range of summation,  and using \eqref{eq:mu_MDP} as well as the fact that for all $m$
$$
\varrho^N(m_N=m)\ge M_{\min}^N \qquad \mbox{where } M_{\min}:=\min_{k=1, \ldots q} M_{1,k}
$$
and therefore $\log\left(\frac{\nu_N}{\rho_N}\right)(m) \le N \log (1/M_{\min})$ for all $m \in \Delta \cap \frac 1N \Z$
we arrive at
\begin{align*}
\int \log\frac {d\nu_N}{d\rho^N}d\mu_N& = \int_{m:\lVert M_J-\hat m\rVert \le c (N^{-1/2+\delta})}
\log\frac {d\nu_N}{d\rho^N}d\mu_N + \int_{m:\lVert M_J-\hat m\rVert > c (N^{-1/2+\delta})}
\log\frac {d\nu_N}{d\rho^N}d\mu_N \\
& \le C + \exp(-N^{2 \delta'}) N \log (1/M_{\min}) \\
& \le C+1
\end{align*}
for $N$ large enough.
In view of \eqref{eq:split} the assertion follows.
\end{proof}

\begin{remark}
Note in \cite{BAZ_chaos} the authors also show increasing propagation of chaos for not necessarily unique, non-degenerate maximizer $M$. Their proof contains a detour from relative entropy to total variation distance and back, which we will omit for simplicity.
\end{remark}

Corollary \ref{cor:usual_chaos} tells us that $k= o(N)$ many spins under the Gibbs measure $\mu_N$ are asymptotically independent, i.e.~
we have (increasing) propagation of chaos. Complementary, the approximation by the mixture $\nu_N$ of Theorem \ref{theorem:main_gen}
shows even more: If $k= \alpha N$ for some $0<\alpha \le 1$,
then the chaos stops to propagate, hence Corollary \ref{cor:usual_chaos} is optimal. The convex
combination of product measures we wish to compare to is given by
$$\rho^{(k)}(\sigma):= \frac{1}{| w|}\sum_{j=1}^p w_j\prod_{i=1}^k\sum_{l=1}^q M_{j,l}\delta_{l}(\sigma_i).$$
Note that if $p=1$, then $\rho^{(k)}=\rho^k$ is a product measure as in Corollary \ref{cor:usual_chaos}.

\begin{corollary}\label{cor:chaos_stops}
 If $\lim_N \frac{k}{N}=\alpha\in(0,1]$, then
$$\liminf_{N\to\infty} H(\mu_N^{(k)}\mid\rho^{(k)})>0.$$
\end{corollary}

\begin{proof}
We will start with the total variation distance of $\mu_N^{(k)}$ and $\nu_N^{(k)}$. Obviously, their total variation distance
is bounded from above by the total variation distance of the full distributions, which in turn
can be bounded by Pinsker's inequality:
\begin{align*}
\lVert \mu_N^{(k)}-\nu_N^{(k)}\rVert_{TV}\le\lVert \mu_N-\nu_N\rVert_{TV}\le \sqrt{2H(\mu_N\mid\nu_N)}.
\end{align*}
The right hand side is $o(1)$ by Theorem \ref{theorem:main_gen}.
Thus, by Pinsker's inequality and the triangle inequality we have
\begin{align*}
\sqrt{2H(\mu_N^{(k)}\mid\rho^{(k)})}\ge \lVert \rho^{(k)}-\nu_N^{(k)}\rVert_{TV}-\lVert \mu_N^{(k)}-\nu_N^{(k)}\rVert_{TV}.
\end{align*}
Since the second term on the right hand side is asymptotically negligible and the first term can only decrease under the push-forward of $m_k$, which yields $\tilde\nu_k$ for the marginal $\nu_N^{(k)}$, we obtain
\begin{align*}
\sqrt{2H(\mu_N^{(k)}\mid\rho^{(k)})}\gtrsim \lVert \rho^{(k)}-\nu_N^{(k)}\rVert_{TV}\ge  \lVert \widetilde{\rho^{(k)}}-\tilde\nu_k\rVert_{TV},
\end{align*}
where $\widetilde{\rho^{(k)}}:=\rho^{(k)}\circ m_k^{-1}$ is a convex combination of rescaled multinomial distributions
\begin{align*}
\widetilde{\rho^{(k)}}(m)= \frac{1}{| w|}\sum_{j=1}^p w_j\binom{N}{Nm_1,\dots,Nm_q}\prod_{l=1}^q M_{j,l}^{Nm_l}.
\end{align*}
By the multivariate CLT, we have
\begin{align*}
\widetilde{\rho^{(k)}}(\sqrt N(\cdot-M_j))\Rightarrow \frac{1}{| w|}\sum_{j=1}^p w_j\mathcal N(0,\diag(M_j)-M_jM_j^T),\\
\tilde \nu_k(\sqrt N(\cdot-M_j))\Rightarrow \frac{1}{| w|}\sum_{j=1}^p w_j\mathcal N(0,\diag(M_j)-M_jM_j^T+\E(\xi\xi^T)).
\end{align*}
By \eqref{eq:Gaussian_distribution}, the covariances strictly differ, hence $\liminf_{N\to\infty}\lVert \widetilde{\rho^{(k)}}-\tilde\nu_k\rVert_{TV}>0$, which implies the claim. Note that we used $M_j\neq M_k$ implicitly, see \cite[Proposition 5.2]{JKLM} wherein an explicit calculation of the limiting total-variation distance can be found.
\end{proof}

\section{Example and extensions: the Curie-Weiss Model}\label{sec:CW}
As mentioned above the Curie-Weiss model falls into our framework above when choosing $q=2$, $\hat m\in[0,1]$ and
$$G(\hat m)=\tfrac \beta 2 (1-2\hat m)^2+h(1-2\hat m)-\hat m\log(\hat m)-(1-\hat m)\log(1-\hat m).$$
For convenience and transparency, let us put down our results in the traditional setup.
Here one chooses the set $\{-1,+1\}$ rather than the set $\{1,\dots,q\}$.
Also, in this scenario, it is reasonable to describe the magnetization not as a vector
$m=(\hat m ,m_2)$ (where $N \hat m$ counts the number of $-1$-spins), but as $ m=1-2\hat m\in[-1,1]$.
The Curie-Weiss model is then retrieved by choosing
$$
F( m)=\frac{\beta}{2} m^2+h m,\qquad  h\ge 0.
$$
For any $h$ (which encodes an external magnetic field)
the corresponding measures are parame\-trized by a positive parameter $\beta>0$ known
as the inverse temperature. Given such $\beta>0$, and an external magnetic field of strength $h>0$,
the Gibbs measure, which is a measure on $\{-1,+1\}^N$ takes the following form
\be
\mu_N(\sigma):= \frac{\exp\left(\frac{\beta}{2N} \sum_{i,j=1}^N \sigma_i \sigma_j+h \sum_{i=1}^N \sigma_i\right)}{Z_N} \qquad \sigma:= (\sigma_i)_{i=1}^N\in \{-1,+1\}^N.
\ee
Again
\be
Z_N=Z_N(\beta)=\sum_{\sigma' \in \{-1,+1\}^N} \exp\left(\frac{\beta}{2N} \sum_{i,j=1}^N \sigma_i' \sigma_j'+h \sum_{i=1}^N \sigma'_i\right)
\ee
is the partition function of the model.

There is a vast literature on the Curie-Weiss model,
from which we know that the model for $h=0$ exhibits a phase transition at $\beta=1$.
This can be seen in different ways: On the one hand, the distribution of the magnetization under the Gibbs measure which can also be written as
$$
\tilde m_N:=\tilde m_N(\sigma):= 
\frac 1N \sum_{i=1}^N \sigma_i,
$$
converges:
\begin{align}\label{eq:mconvergence}
\tilde\mu_N:=\mu_N\circ \tilde m_N^{-1} \Rightarrow \frac 12 \left(\delta_{m^+(\beta)}+ \delta_{m^-(\beta)}\right) \qquad \mbox{if }h=0.
\end{align}
and
\begin{align}\label{eq:mconvergence_h}
\tilde\mu_N=\mu_N\circ \tilde m_N^{-1} \Rightarrow \frac 12 \left(\delta_{m^+(\beta,h)}+ \delta_{m^-(\beta,h)}\right) \qquad \mbox{if }h\neq 0.
\end{align}
Here, $\Rightarrow$ denotes weak convergence, $\delta$ is the Dirac measure and $m^+:=m^+(\beta,h)$ is the largest solution of the equation
\begin{align}\label{eq:tanh}
z=\tanh(\beta z+h).
\end{align}
This, and other results on the Curie-Weiss model can be found e.g.\
in the textbooks, \cite{BovierSMoDS},
\cite{EllisEntropyLargeDeviationsAndStatisticalMechanics}, \cite{Velenik_book}.

Let us first specify to the situation with zero external field, where $h=0$.
Then $m^+(\beta):= m^+(\beta,0)=-m^-(\beta)$.
For $\beta\le 1$, the high temperature regime, $m^+(\beta)$ and
$m^{-}(\beta)$, agree, while for the low temperature regime $\beta>1$, $m^+(\beta)\neq m^-(\beta)$.

As mentioned at $\beta<1$ and $h=0$, we have a unique solution of \eqref{eq:tanh} at $z=0$,
which in the notation of Section \ref{sec:General} corresponds to $p=1$ with $M_1=1/2$.
Moreover we have $H_1=-G''(1/2)=4(1-\beta)$,
thus the distribution of $\xi_1$ before truncation is $\mathcal N (0,\frac {\beta} {4(1-\beta)})$ and still $\xi_2=-\xi_1$.
Let $\tilde \xi \sim\mathcal N (0,\frac {\beta} {(1-\beta)})$ and define its truncation by $\xi=\tilde \xi \mathbbm 1 _{[-N^{\delta},+N^{\delta}]}$ for some $0<\delta<\frac 1 6$ (which equals $2\xi_2$ in the notation of the previous sections). Then \eqref{eq:nu} simplifies to
\begin{align*}
\nu_N(\sigma)&:= \E_\xi\Big[\prod_{i=1}^N \Big( \frac 1 2+\frac{\xi_1}{\sqrt N}\Big)\mathbbm 1_{-1}(\sigma_i)+ \Big(\frac 1 2+\frac{\xi_2}{\sqrt N}\Big)\mathbbm 1_{+1}(\sigma_i)\Big]\\
&=2^{-N}\E_\xi\Big[\prod_{i=1}^N \big( 1+\sigma_i\frac{\xi}{\sqrt N}\big)\Big]
\end{align*}
and we obtain as a corollary of Theorem \ref{theorem:main_gen}
\begin{corollary}\label{cor:main1a}
Assume that $h=0,\beta<1$ and define $\xi_1$ and let $(\xi_N)$ be its truncation as in Theorem \ref{theorem:main_gen}.
Moreover, also define $\nu_N$ as above. Then
\begin{align*} 
H(\mu_N|\nu_N)\to 0 \qquad \mbox{as well as} \quad H(\nu_N|\mu_N)\to 0.
\end{align*}
\end{corollary}

The analogue of this statement for total variation distance and $\xi$ being Beta-distributed is given by \cite[Theorem 3.5]{JKLM}.

As is well known for $h=0$ and $\beta >1$ we have three solutions of \eqref{eq:tanh}. The one with $z=0$ corresponds to a minimum of the free energy, while at $z=m^+$ and $z=m^-
=-m^+$ there are maxima of the free energy.

Again using the notation of Section \ref{sec:General} this means that in this case $p=2$ with
$M_1=(\frac{1+m^+}2,\frac{1-m^+}2)$
and $M_2=(\frac{1-m^+}2,\frac{1+m^+}2)$. In this case one computes that
\begin{align}\label{eq:CW_LT_H}
H_2=H_1=-G''\left(\frac{1+m^+}2\right)=4\frac{1-\beta(1-(m^+)^2)}{1-(m^+)^2},
\end{align}
such that $\xi_1$ before truncation is $\mathcal N (0,r^2)$-distributed, where
\begin{align}\label{eq:CW_LT_var}
r^2=\frac 1 4\left(\frac{1-(m^+)^2}{1-\beta(1-(m^+)^2)}+(m^+)^2-1\right)=\frac{\beta\big(1-(m^+)^2\big)^2}{4-4\beta\big(1-(m^+)^2\big)}.
\end{align}
Again $\xi_2=-\xi_1$.
Then, as a corollary of Theorem \ref{theorem:main_gen} we obtain
\begin{corollary}\label{cor:main1c}
If $h=0,\beta>1$, define $\xi_1$ as above and let $(\xi_N)$ be its truncation as in Theorem \ref{theorem:main_gen}.
Moreover, for all $N$ let
\begin{equation} \label{eq:nuLT}
\nu_N(\sigma):= \frac 12\E_{\xi_N}  \left[\prod_{i=1}^N  \left(\frac{1+\sigma_i m^+}2+\sigma_i\frac{\xi_N}{\sqrt N}\right)+ \prod_{i=1}^N  \left(\frac{1-\sigma_i m^+}2+\sigma_i\frac{\xi_N}{\sqrt N}\right)\right]
\end{equation}

Then
\begin{equation} \label{eq:convLT2}
H(\mu_N|\nu_N)\to 0 \qquad \mbox{as well as}\quad H(\nu_N|\mu_N)\to 0.
\end{equation}
\end{corollary}

The situation for $h \neq 0$, w.l.o.g. $h>0$, is similar to the situation where $\beta <1$ and $h=0$. There are either one or three solutions to \eqref{eq:tanh}, of which
the strictly positive one, called $m^+$ is the unique global maximum of the free energy. Then, in language of Theorem \ref{theorem:main_gen}
we are in the situation where $p=1$ and $M_1=(\frac{1+m^+}2,\frac{1-m^+ }2)$. Therefore, the values of $H_1$ are as in \eqref{eq:CW_LT_H}, the variance $r^2$ is given in \eqref{eq:CW_LT_var} and Theorem \ref{theorem:main_gen} implies the following approximation for the Curie-Weiss model with external field.
\begin{corollary}\label{cor:main1d}
If $h>0$ and $\beta$ is arbitrary, define $\xi_1$ as above and let $(\xi_N)$ again be its truncation as in Theorem \ref{theorem:main_gen}.
Moreover, for all $N$ let
\begin{align*}
\nu_N(\sigma):&= \E_\xi\Big[\prod_{i=1}^N \Big(\frac{1+m^+}2+\frac{\xi_N}{\sqrt N}\Big)\mathbbm 1_{+1}(\sigma_i)+ \Big(\frac{1-m^+ }2-\frac{\xi_N}{\sqrt N}\Big)\mathbbm 1_{-1}(\sigma_i)\Big]\\
&=\E_\xi\Big[\prod_{i=1}^N \frac{1+\sigma_i m^+}2+\sigma_i\frac{\xi_N}{\sqrt N}\Big]
\end{align*}
Then
\begin{equation} \label{eq:convLT}
H(\mu_N|\nu_N)\to 0 \qquad \mbox{as well as}\quad H(\nu_N|\mu_N)\to 0.
\end{equation}
\end{corollary}

Finally, notice that Theorem \ref{theorem:main_gen} does not cover the critical case of case of $\beta=1, h=0$ in which case there
is a unique maximum of $G$ at $M_1=1/2$, but this minimum is degenerate, i.e.\ $G''(0)=0$.
While such a situation is technically cumbersome to treat in general (see Remark \ref{rem:generalization}), 
in the case of the Curie-Weiss model we do have
a version of Theorem \ref{theorem:main_gen} at $\beta=1, h=0$, due to its explicitly given description.

Here we are inspired by the fact that the
phase transition in the Curie-Weiss model is also visible on the level of fluctuations:
While for $h=0$, $\beta<1$ in the limit $N \to \infty$, $\sqrt N \tilde m_N$ has Gaussian fluctuations under $\mu_N$ with expectation $0$ and variance $1/(1-\beta)$, of course, this cannot be true for $h=0,\beta=1$. In this
case, the distribution of $N^{1/4} \tilde m_N$ has a limiting density
\be
f_1(x)=\frac{\exp(-\frac 1{12} x^4)}{\int_\R \exp(-\frac 1{12} y^4)dy}
\ee
(see e.g.\ \cite{Ellis_Newman_78b}).
This is not only interesting in its own right, it also hints at which mixture of product measures we should compare the Curie-Weiss distribution to, when $\beta=1$ and $h=0$:
\begin{theorem}\label{theo:mainCW1b}
If $h=0,\beta=1$, let $(\xi_N)$ be a sequence of random variables with $\xi_N:=\xi \ind_{|\xi|\le N^\delta}$ for some small enough $0<\delta<\frac 1{20}$ where $\xi$ is a random
variable that has a distribution with density $f_1$.
Moreover, for all $N$ let
\be \label{eq:nucritT}
\nu_N(\sigma):= \E_{\xi_N} \prod_{i=1}^N \frac 12 \left(1+\frac{\xi_N \sigma_i}{\sqrt[4] N}\right)
\ee
Then
\be \label{eq:convcritT}
H(\mu_N|\nu_N)\to 0 \qquad \mbox{as well as}\quad H(\nu_N|\mu_N)\to 0.
\ee
\end{theorem}

\begin{proof}
In the notation of this section $m\in[-1,1]$, we have
\begin{align}\label{eq:quotient}
\frac{\tilde \nu_N(m)}{\tilde \mu_N(m)}&=
\frac{Z_N 
\E_{\xi_N} \left[\left(1+\frac{\xi_N}{\sqrt[4] N}\right)^{\frac{1+m}2N}\left(1-\frac{\xi_N}{\sqrt[4] N}\right)^{{\frac{1-m}2N}}\right]}
{2^{N} \exp\left(\frac 1{2} Nm^2\right)}
\end{align}
Note that the number of positive spins is a sufficient statistic for $\frac{\nu_N}{\mu_N}$ and $\frac{\mu_N}{\nu_N}$, a fact which we will use later. We will follow the route of the proof of Section \ref{sec:General} (In particular we show analogues of Lemma \ref{lem:partitionfunction}), Lemma \ref{lem:dichtequotient} and of the proof of Theorem \ref{theorem:main_gen}).

Let us first compute the well-known asymptotics of $Z_N$ for $\beta=1$, which can be found for instance in \cite{M-L82}, \cite{bolthausen_laplaceii}, and \cite{BB90}. Following the route of Lemma \ref{lem:partitionfunction}, we now provide a self-contained argument.
Recall that $G(\hat m)=\tfrac 1 2 (1-2\hat m)^2-\hat m\log(\hat m)-(1-\hat m)\log(1-\hat m)$, $\hat m=(1-m)/2\in [0,1]$, whose critical point is $M=1/2$ with value $G(1/2)=\log 2$. The first three derivatives at $1/2$ vanish and $G^{(4)}(1/2)=-32$. Similar to the proof of Lemma \ref{lem:partitionfunction}, we apply Laplace's method (now, of fourth order)
\begin{align*}
Z_N&\sim\int_0^1\exp\big(NG(x)\big)\frac{N^{(q-1)/2}}{\sqrt{(2\pi)^{q-1} \prod_{k=1}^q x_k }} dx\\
&=\int_0^1\exp\big(NG(1/2)+\frac N {4!}G^{(4)}(1/2)(x-1/2)^4+No(x-1/2)^4\big)\frac{N^{1/2}}{\sqrt{2\pi x(1-x) }} dx \\
&\sim\frac{2^{N}}{\sqrt{ 2 \pi}}  N^{\frac 14} \int_{-\infty}^\infty e^{-\frac 1 {12} u^4} du=c2^NN^{1/4}
\end{align*}
where we changed variables to $u=(2x-1)N^{1/4}$ and denoted $c=3^{1/4}\Gamma(1/4)/\sqrt \pi$.
Similar to \eqref{eq:dichtequotientexpexp} we expand the logarithm in a Taylor expansion (now to fourth order and using $N(\xi_N /\sqrt[4]{N})^5\to 0$), yielding
\begin{align}\label{eq:Dichtequotient_CWcrit}
    \frac {\tilde\nu_N}{\tilde \mu_N}(m)&=
\frac{Z_N 
\E_{\xi_N} \left[\left(1+\frac{\xi_N}{\sqrt[4] N}\right)^{\frac{1+m}2N}\left(1-\frac{\xi_N}{\sqrt[4] N}\right)^{{\frac{1-m}2N}}\right]}
{2^{N} \exp\left(\frac 1{2} Nm^2\right)}\nonumber\\
&\sim cN^{1/4}e^{-\tfrac 1 2 Nm^2}\E_{\xi_N} \Bigg[\exp\left(\frac{1+m}2N\left(\frac{\xi_N}{N^{1/4}}-\frac{\xi_N^2}{2N^{1/2}}+\frac{\xi_N^3}{3N^{3/4}}-\frac{\xi_N^4}{4N} \right)\right)\nonumber\\
&\exp\left(\frac{1-m}2N\left(-\frac{\xi_N}{N^{1/4}}-\frac{\xi_N^2}{2N^{1/2}}-\frac{\xi_N^3}{3N^{3/4}}-\frac{\xi_N^4}{4N} \right)\right)\Bigg]\nonumber\\
&=ce^{-\frac 1 2 Nm^2}\int_{-N^\delta}^{N^\delta} \exp\left(N^{3/4}mx-\frac{N^{1/2}}2x^2+\frac{N^{1/4}m}{3}x^3-\frac{x^4}{4}\right)f_1(x)N^{1/4}dx\nonumber\\
&=\frac 1{\sqrt{2\pi}}\int_{-N^{1/4+\delta}}^{N^{1/4+\delta}} \exp\left(-\frac{1}2(y-N^{1/2}m)^2+\frac{m}{3N^{1/2}}y^3-\frac{y^4}{4N}-\frac{y^4}{12N}\right)dy,
\end{align}
where in the last step we plugged in the explicit density $f_1$, completed the square and changed variables to $y=xN^{1/4}$. We will only need to consider $|m|<N^{-1/4+\delta/2}$, for which the integrand converges pointwise to a shifted Gaussian (and the shift stays in the interior of the integration area). Analogously to \eqref{eq:quadr_erg}, we obtain by dominated convergence
\begin{align}\label{eq:criticalDichte}
    \frac {\tilde\nu_N}{\tilde \mu_N}(m)= 1+o(1)
    \end{align}
uniformly in $|m|<N^{-1/4+\delta/2}$.

Now we can finish the proof by a moderate deviations argument, similar to \eqref{eq:mu_MDP}.
Indeed, Example 2.1 of \cite{EL04} (with $\alpha=7/8$) states that $N^{1/8}\tilde m_N$ satisfies a MDP with speed $N^{1/2}$ and rate function $\tfrac 1 {12}x^4$, hence for $A=\R\setminus (-N^{-1/8+\delta/2},+N^{-1/8+\delta/2})$ it holds
\begin{align*}
\tilde \mu_N(|m|\ge N^{-1/4+\delta/2})\le \exp\big(-N^{1/2}\inf_{x\in A }\tfrac 1 {12} x^4\big)= \exp(-\tfrac 1 {12}N^{2\delta}).
\end{align*}
Moreover, the calculation \eqref{eq:Dichtequotient_CWcrit} still implies a very rough bound on the density: For some constants $c,C>0$ we have
\begin{align*}
\exp(-cN^2) < \frac {\tilde\nu_N(m)}{\tilde \mu_N(m)}< \exp(CN^2).
\end{align*}
We conclude that
\begin{align*}
H(\tilde \mu_N|\tilde \nu_N)&=\int_{|m|< N^{-\frac 14+\frac\delta 2}}\log \left(\frac{\tilde \mu_N(m)}{\tilde \nu_N(m)}\right)d\tilde \mu_N(m)+
\int_{|m|\ge N^{-\frac 14+\frac\delta 2}} \log \left(\frac{\tilde \mu_N(m)}{\tilde \nu_N(m)}\right)d\tilde \mu_N(m)\to 0.
\end{align*}
In the same way as in \eqref{eq:bound_nu_tilde}, one may also verify that $\tilde \nu_N(|m|\ge N^{-1/4+\delta/2})\le \exp(- cN^{2\delta})$, which implies the claim $H(\tilde \nu_N|\tilde \mu_N)\to 0$.
\end{proof}

\section*{Acknowledgement}
The authors are funded by the Deutsche Forschungsgemeinschaft (DFG, German Research Foundation) under Germany's Excellence Strategy EXC 2044 - 390685587, Mathematics M\"unster: \emph{Dynamics-Geometry-Structure}. JJ and ZK have been supported by the DFG priority program SPP 2265 \emph{Random Geometric Systems}.


\begin{thebibliography}{10}

\bibitem{BAYE23}
Y.~Barhoumi-Andr{\'e}ani, M.~Butzek, and P.~Eichelsbacher.
\newblock A surrogate by exchangeability approach to the {C}urie-{W}eiss model.
\newblock {\em arXiv preprint arXiv:2305.06872}, 2023.

\bibitem{BB90}
G.~Ben~Arous and M.~Brunaud.
\newblock M\'{e}thode de {L}aplace: \'{e}tude variationnelle des fluctuations
  de diffusions de type ``champ moyen''.
\newblock {\em Stochastics Stochastics Rep.}, 31(1-4):79--144, 1990.

\bibitem{BAZ_chaos}
G.~Ben~Arous and O.~Zeitouni.
\newblock Increasing propagation of chaos for mean field models.
\newblock {\em Ann. Inst. H. Poincar\'{e} Probab. Statist.}, 35(1):85--102,
  1999.

\bibitem{bolthausen_laplaceii}
E.~Bolthausen.
\newblock Laplace approximations for sums of independent random vectors. {II}.
  {D}egenerate maxima and manifolds of maxima.
\newblock {\em Probab. Theory Related Fields}, 76(2):167--206, 1987.

\bibitem{BovierSMoDS}
A.~Bovier.
\newblock {\em {Statistical Mechanics of Disordered Systems - A Mathematical
  Perspective}}.
\newblock Cambridge Series in Statistical and Probabilistic Mathematics, 2006.

\bibitem{csiszar}
I.~Csisz\'{a}r.
\newblock Sanov property, generalized {$I$}-projection and a conditional limit
  theorem.
\newblock {\em Ann. Probab.}, 12(3):768--793, 1984.

\bibitem{FdHLarge}
F.~den Hollander.
\newblock {\em Large deviations}, volume~14 of {\em Fields Institute
  Monographs}.
\newblock American Mathematical Society, Providence, RI, 2000.

\bibitem{EL04}
P.~Eichelsbacher and M.~L\"{o}we.
\newblock Moderate deviations for a class of mean-field models.
\newblock {\em Markov Process. Related Fields}, 10(2):345--366, 2004.

\bibitem{EllisEntropyLargeDeviationsAndStatisticalMechanics}
R.~S. Ellis.
\newblock {\em Entropy, large deviations, and statistical mechanics}.
\newblock Classics in Mathematics. Springer-Verlag, Berlin, 2006.
\newblock Reprint of the 1985 original.

\bibitem{Ellis_Newman_78b}
R.~S. Ellis and C.~M. Newman.
\newblock Limit theorems for sums of dependent random variables occurring in
  statistical mechanics.
\newblock {\em Z. Wahrsch. Verw. Gebiete}, 44(2):117--139, 1978.

\bibitem{Velenik_book}
S.~Friedli and Y.~Velenik.
\newblock {\em Statistical mechanics of lattice systems}.
\newblock Cambridge University Press, Cambridge, 2018.

\bibitem{JKLM}
J.~Jalowy, Z.~Kabluchko, M.~L{\"o}we, and A.~Marynych.
\newblock When does the chaos in the {C}urie-{W}eiss model stop to propagate?
\newblock {\em Electronic Journal of Probability}, 28:1--17, 2023.

\bibitem{KestenSchonmann89}
H.~Kesten and R.~H. Schonmann.
\newblock Behavior in large dimensions of the {P}otts and {H}eisenberg models.
\newblock {\em Rev. Math. Phys.}, 1(2-3):147--182, 1989.

\bibitem{KirschSurvey}
W.~Kirsch.
\newblock A survey on the method of moments.
\newblock {\em Preprint, book in preparation}, 2015.

\bibitem{LLP10}
D.~A. Levin, M.~J. Luczak, and Y.~Peres.
\newblock Glauber dynamics for the mean-field {I}sing model: cut-off, critical
  power law, and metastability.
\newblock {\em Probab. Theory Related Fields}, 146(1-2):223--265, 2010.

\bibitem{LST07}
T.~M. Liggett, J.~E. Steif, and B.~T{\'o}th.
\newblock Statistical mechanical systems on complete graphs, infinite
  exchangeability, finite extensions and a discrete finite moment problem.
\newblock 2007.

\bibitem{M-L82}
A.~Martin-L\"{o}f.
\newblock A {L}aplace approximation for sums of independent random variables.
\newblock {\em Z. Wahrsch. Verw. Gebiete}, 59(1):101--115, 1982.

\bibitem{p-spin}
S.~Mukherjee, J.~Son, and B.~B. Bhattacharya.
\newblock Fluctuations of the magnetization in the $p$-spin {C}urie--{W}eiss
  model.
\newblock {\em Communications in Mathematical Physics}, 387(2):681--728, 2021.

\end{thebibliography}

\end{document}